\documentclass{amsart}
\usepackage{ifpdf}
\usepackage{amsmath, amsthm, amsfonts, amssymb, amscd}
\usepackage[active]{srcltx}
\usepackage[pdftex]{graphicx}
\usepackage{color} 
    \ifpdf

      \usepackage[pdftex]{graphicx}  

      \usepackage[pdftex]{hyperref}

    \else

      \usepackage[dvips]{graphicx}  

      \newcommand{\href}[2]{#2}

    \fi

\newcommand{\abs}[1]{\left\lvert{#1}\right\rvert}
\newcommand{\norm}[1]{\left\|{#1}\right\|}

\DeclareMathOperator{\SL}{\rm{SL}}
\DeclareMathOperator{\inter}{\rm{int}}
\DeclareMathOperator{\bd}{\partial}
\DeclareMathOperator{\cl}{cl}

\DeclareMathOperator{\diff}{\rm{Diff}}

\newcommand{\ol}{\overline}
\renewcommand{\hat}{\widehat}
\newcommand{\til}{\widetilde}

\newcommand{\R}{\mathbb{R}}\newcommand{\N}{\mathbb{N}}
\newcommand{\Z}{\mathbb{Z}}
\newcommand{\T}{\mathbb{T}}
\newcommand{\D}{\mathbb{D}}
\renewcommand{\SS}{\mathbb{S}}

\newcommand{\sm}{\setminus}

\newcommand{\id}{\mathrm{Id}}

\newcommand{\Leb}{\mathrm{Leb}}

\newcommand{\ie}{i.e.\ }

\newtheorem{theorem}{Theorem}

\newtheorem{lemma}[theorem]{Lemma}
\newtheorem{proposition}[theorem]{Proposition}

\newtheorem*{theorem*}{Theorem}

\theoremstyle{definition}

\theoremstyle{remark}

\title[Irrotational diffeomorphisms with sublinear diffusion]
{Area-preserving irrotational diffeomorphisms of the torus with sublinear diffusion}
\author{Andres Koropecki}
\address{Universidade Federal Fluminense, Instituto de Matem\'atica e Estat\'\i stica, Rua M\'ario Santos Braga S/N, 24020-140 Niteroi, RJ, Brasil}
\email{ak@id.uff.br}
\author{Fabio Armando Tal}
\address{Instituto de Matem\'atica e Estat\'\i stica, Universidade de S\~ao Paulo, Rua do Mat\~ao 1010, Cidade Universit\'aria, 05508-090 S\~ao Paulo, SP, Brazil}
\email{fabiotal@ime.usp.br}
\thanks{The first author was partially supported by CNPq-Brasil. The second author was partially supported by FAPESP and CNPq-Brasil}

\begin{document}

\begin{abstract} We construct a $C^\infty$ area-preserving diffeomorphism of the two-dimensional torus which is Bernoulli (in particular, ergodic) with respect to Lebesgue measure, homotopic to the identity, and has a lift to the universal covering whose rotation set is $\{(0,0)\}$, which in addition has the property that almost every orbit by the lifted dynamics is unbounded and accumulates in every direction of the circle at infinity.
\end{abstract}

\maketitle

\section{Introduction}
Consider the torus $\T^2 = \R^2/\Z^2$ with covering projection $\pi\colon \R^2\to \Z^2$, an area-preserving homeomorphism $f\colon \T^2\to \T^2$ homotopic to the identity, and a lift $\hat{f}\colon \R^2\to \R^2$ of $f$. The \emph{rotation set} $\rho(\hat{f})$ of $\hat{f}$, introduced by Misiurewicz and Ziemian in \cite{m-z} as a generalization of the rotation number of an oritentation-preserving circle homeomorphism, is defined as the set of all limits of sequences of the form
$$v = \lim_{k\to \infty} (\hat{f}^{n_k}(z_k)-z_k)/n,$$
where $(n_k)_{k\in \N}$ is a sequence of integers such that $n_k\to \infty$ as $k\to \infty$. Roughly speaking, this set measures the average asymptotic rotation of orbits. It is known that $\rho(\hat{f})$ is always compact and convex, and if $v\in \rho(\hat{f})$, is extremal or interior in $\rho(\hat{f})$, then there is $z\in \R^2$ such that the pointwise rotation vector
\begin{equation}\rho(\hat{f},z)=\label{eq:lim} \lim_{n\to \infty} (\hat{f}^n(z)-z)/n \end{equation}
exists and coincides with $v$.
Moreover, the Ergodic Theorem implies that the limit (\ref{eq:lim}) exists for almost every $z\in \R^2$.

Another interpretation of the rotation set is as a measure of the \emph{linear} rate of diffusion of orbits in the universal covering: how fast and in which directions do orbits drift away. One should emphasize the ``linear'' part, because points that separate from each other at a sublinear rate (e.g. with distance increasing as $\sqrt{n}$) are not differentiated in terms of their rotation vectors. In particular, from the perspective of tthe rotation set, orbits drifting away from the origin at a sublinear rate are undistinguishable from fixed points.

If $\rho(\hat{f}) = \{(0,0)\}$, we say  $f$ is \emph{irrotational} and $\hat{f}$ is its irrotational lift. In this case, it is easy to see that $\rho(\hat{f},z)$ exists for \emph{all} $z\in \R^2$ and is equal to $(0,0)$ (moreover, the limit (\ref{eq:lim}) converges uniformly).
A natural question that arises is whether an irrotational homeomorphism can exhibit sublinear diffusion. For instance, it is well-known and easy to verify that there is no such phenomenon in dimension one: if $F\colon \R\to \R$ is a lift of an orientation preserving homeomorphism of $\T^1=\R/\Z$ and its rotation number is $0$, then $\abs{\smash{F^n(x)-x}}$ is uniformly bounded (a property known as \emph{uniformly bounded deviations}).
A similar property was proved in \cite{ABT} for homeomorphisms of $\T^2$ in the homotopy class of a Dehn twist $(x,y)\mapsto (x+y,x)$.

It is conjectured that the same property holds in the horizontal direction if one assumes that $\rho(\hat{f})$ is an interval such as $\{0\}\times [a,b]$, with $a<0<b$: namely, that $\hat{f}^n(z)-z$ is uniformly bounded in the horizontal direction. Progress in this direction has been made by D\'avalos \cite{davalo}.

In this article we provide an example of an area-preserving $C^\infty$ ergodic diffeomorphism which is irrotational but exhibits sublinear diffusion for almost every point and in all possible directions. Before stating our main theorem, let us give a definition: given a set $X\subset \R^2$, we say that $X$ accumulates in the direction $v\in \SS^1$ at infinity if there is a sequence $\{x_n\}_{n\geq 0}$ in $X$ such that 
$$\lim_{n\to \infty}\norm{x_n} = \infty \quad \text{ and }\quad \lim_{n\to\infty} \frac{x_n-x_0}{\norm{x_n-x_0}}=v.$$
The boundary of $X$ at infinity is defined as the set $\bd_\infty X\subset \SS^1$ consisting of all $v\in \SS^1$ such that $X$ accumulates in the direction $v$ at infinity.

\begin{theorem} \label{th:example} There is a $C^\infty$ area-preserving diffeomorphism $f\colon \T^2\to \T^2$ homotopic to the identity, with a lift $\hat{f}\colon \R^2\to \R^2$ such that
\begin{itemize}
\item $\rho(\hat{f})=\{(0,0)\}$ (\ie $f$ is irrotational and $\hat{f}$ is the irrotational lift);
\item $f$ is metrically isomorphic to a Bernoulli shift (in particular, $f$ is ergodic) with Lebesgue measure;
\item For Lebesgue almost every point $\hat{x}\in \R^2$, the forward and backward orbits of $\hat{x}$ accumulates in every direction at infinity, \ie $$\bd_\infty \{\hat{f}^n(\hat{x}):n\in \N\}= \SS^1=\bd_\infty\{\hat{f}^{-n}(\hat{x}):n\in \N\}.$$
Moreover, the forward and backward orbits of $\hat{x}$ visit every fundamental domain $[0,1]^2+v$, with $v\in \Z^2$.
\end{itemize} 
\end{theorem}

The construction is based on a simple topological idea of embedding an open topological disk $U\subset \T^2$ in a wild way so that $U$ has full Lebesgue measure and any lift of $U$ to the universal covering accumulates in every direction at infinity:
\begin{theorem} \label{th:weirdopen} There exists an open set $\hat{U}\subset \R^2$ such that 
\begin{enumerate}
\item $\hat{U}$ is connected and simply connected;
\item $U=\pi(\hat{U})$ is simply connected, or equivalently, $\hat{U}$ is disjoint from $\hat{U}+v$ for all $v\in \Z^2$, $v\neq (0,0)$;
\item $\hat{U}$ intersects every square $[0,1]^2 + v$, with $v\in \R^2$ (in particular, $\hat{U}$ accumulates in every direction at infinity).
\item $U$ has full Lebesgue measure in $\T^2$.
\end{enumerate}
\end{theorem}

To construct the example from Theorem \ref{th:example}, we consider the open set $U$ from the previous theorem and we ``glue'' in $U$ a diffeomorphism of the open unit disk which is Bernoulli and which extends to the closure of $U$ as the identity. The later step can be done by a slight modification of results due to Katok \cite{katok}.

It is important to emphasize that the example from Theorem \ref{th:example} is rather pathological. In particular, its fixed point set is very large, as it contains the boundary of $U$, which is a topologically rich continuum carrying the homology of $\T^2$ (what in \cite{essential} is called a \emph{fully essential continuum}). The topological richness is illustrated by the following Wada-type property in the universal covering: if $\hat{K} = \bd \hat{U}$ and we assume that $\inter(\cl(\hat{U})) = \hat{U}$ (which we can) then $\hat{K}$ is contained in the fixed point set of $\hat{f}$, invariant by $\Z^2$-translations, and further, $\hat{K}$ is the boundary of each of the sets $\hat{U}+v$, with $v\in \Z^2$ (each of which is a connected component of the complement of $\hat{K}$). 

In \cite{irrotational} the authors prove that this is necessarily the case for such an example (\ie an example which has unbounded orbits accumulating in more than one direction at infinity), and in particular no such example exists in the real analytic setting.

Using the terminology of \cite{essential}, Theorem \ref{th:example} also provides an example of a non-annular nonwandering diffeomorphism with an invariant unbounded disk, showing that all cases described in the main theorem of said article are possible even in the $C^\infty$ setting.

It may be interesting to note that, using the same technique, one can obtain the following:
\begin{theorem}\label{th:bounded} There is a $C^\infty$ area-preserving irrotational diffeomorphism $f\colon \T^2\to \T^2$ homotopic to the identity, with a lift $\hat{f}\colon \R^2\to \R^2$ such that 
\begin{enumerate}
\item[(1)] Every orbit of $\hat{f}$ is bounded, but
\item[(2)] for every finite set $\{v_1,\dots, v_n\}\subset \Z^2$, there is $z\in [0,1]^2$ such that the orbit of $z$ intersects $[0,1]^2+v_i$ for all $i\in\{1,\dots n\}$.
\end{enumerate}
\end{theorem}
In particular, in the above example all orbits are bounded but the deviations $\hat{f}^n(z)-z$ are not uniformly bounded in any direction.

\subsection*{Acknowledgements} The authors would like to thank P. Boyland for his interesting remarks.

%

\section{Notations}
\label{sec:quasiconvex-index}
We use the notation $\Z^2_*$ to represent the set $\Z^2\sm \{(0,0)\}$.  By a compact arc in $X$ we mean a continuous map $\gamma\colon[0,1]\to X$, and $[\gamma]$ denotes its image $\gamma([0,1])$.  We denote by $\omega = dx\wedge dy$ the usual volume form in $\R^2$, by $\Leb$ the measure induced by $\omega$ (\ie Lebesgue measure) and by $\Leb_{\T^2}$ the corresponding probability measure on $\T^2$. We say that $f\colon \T^2 \to \T^2$ is area-preserving if $f_*(\Leb_{\T^2}) = \Leb_{\T^2}$. If $f$ is a diffeomorphism, this is the same as saying that $f^*(\omega')=\omega'$, where $\omega'$ is the volume form induced in $\T^2$ by $\omega$.

\section{Theorem \ref{th:weirdopen}: construction of the unbounded disk $U$.}
\label{sec:weirdopen}


%
%
In this section we will prove Theorem \ref{th:weirdopen}. The proof is somewhat similar to the construction given in \cite{anosov-2} of an embedded line $[0,\infty)\to \T^2$ which lifts to a line in $\R^2$ accumulating on all directions at infinity (see \cite{anosov-survey} for a comprehensive survey on the limit behavior at infinity of the lift to the universal covering of injective curves on surfaces). In particular, combining the example of \cite{anosov-2} with some results from \cite{anosov-survey}, one can obtain an open set satisfying properties (1)-(3) from Theorem \ref{th:weirdopen}. However, the proof is somewhat indirect and we are not able to obtain property (4), which is necessary for our main result, so we will give an explicit construction.


We will use the following simple observation:
\begin{lemma}\label{lem:ex-aux} If $\hat{x}$ and $\hat{y}$ are points in $\R^2$ with $\pi(\hat{x})\neq \pi(\hat{y})$, then there is an arc $\gamma$ joining $\hat{x}$ to $\hat{y}$ such that $\pi\circ \gamma$ is injective.
\end{lemma}

\begin{proof} We need to find a simple arc $\gamma$ joining $\hat{x}$ to $\hat{y}$ such that $[\gamma]$ is disjoint from $[\gamma]+v$ for all nonzero $v\in \Z^2$. 
Let $L$ be the straight segment joining $\hat{x}$ to $\hat{y}$. If $L$ is disjoint from $L+v$ for all nonzero $v\in \Z^2$ (for example if $L$ has irrational slope), then a parametrization of $L$ gives the required arc. Now suppose that $L\cap (L+(p,q))\neq \emptyset$ for some $(p,q)\in \Z^2$. 
Without loss of generality, we may assume that $\hat{x}+(p,q)\in L$, and $p$ and $q$ are coprime, so that there are integers $a,b$ such that $ap+bq=1$. Using a linear change of coordinates in $\SL(2,\Z)$ given by
$$A =  \left(
    \begin{matrix}
      a & b \\ -q & p
    \end{matrix}
  \right).$$
we reduce the problem to the case $v=(1,0)$. In that case, we can write $\hat{x} = (a, c)$ and $\hat{y}=(b,c)$, and $L$ is a horizontal segment. Assume for simplicity that $a<b$, and define explicitly, for $t\in [0,1]$,
$$\gamma(t) = \left(a+t(b-a),\, c+\frac{1}{4}\abs{\sin(\pi t(b-a))} - \frac{t}{4}\abs{\sin(\pi (b-a))}\right).$$
It is clear that if $(v_1,v_2)\in \Z^2$ with $v_2\neq 0$ then $[\gamma]$ is disjoint from $[\gamma]+(v_1, v_2)$ (because the vertical width of $\gamma$ is less than $1/2$). It remains to show that $[\gamma]+(v_1,0)$ is disjoint from $[\gamma]$ if $0\neq v_1\in \Z$. Suppose for contradiction that $\gamma(s) = \gamma(t)+v_1$. Equating the first coordinates we see that $s=t+v_1/(b-a)$, while equating the second coordinates and using the fact that $v_1\in \Z$ one finds that $(v_1/(b-a))\abs{\sin(\pi(b-a))} =0$. This is only possible if $v_1=0$, because $b-a$ is not an integer due to the fact that $\hat{y}-\hat{x}\notin \Z^2$. This completes the proof.
\end{proof}

\begin{proof}[Proof of Theorem \ref{th:weirdopen}]

We first observe that we can replace item (4) by a weaker condition, namely that $\pi(\hat{U})$ is dense in $\T^2$. To see this, suppose that we have found an open simply connected set $\hat{U}$ satisfying items (1)-(3) of the theorem, and such that $U=\pi(\hat{U})$ is dense in $\T^2$. The normalized Lebesgue measure in $U$, \ie the Borel measure defined by $\mu(E) = \Leb_{\T^2}(E\cap U)/\Leb_{\T^2}(U)$, is positive on open sets and non-atomic, so by a classical result of Oxtoby and Ulam \cite{oxtoby-ulam}, there is a homeomorphism $h\colon \T^2\to \T^2$ such that $h_*(\mu) = \Leb_{\T^2}$, so $h(U)$ has full Lebesgue measure. Furthermore, by Theorem 2 of \cite{oxtoby-ulam}, $h$ can be chosen such that it lifts to a homeomorphism $\hat{h}\colon \R^2\to \R^2$ which leaves the boundary of the unit square $[0,1]^2$ pointwise fixed. 

This means that $h$ is homotopic to the identity and $\hat{h}$ leaves invariant any square $[0,1]^2+v$, with $v\in \Z^2$. In particular, since $\hat{U}$ intersects every square $[0,1]^2+v$ with $v\in \Z^2$, so does $\hat{U}'=\hat{h}(\hat{U})$.  Since $\hat{U}'$ is clearly connected, simply connected, and disjoint form its translates by elements of $\Z^2_*$, and $\pi(\hat{U}')= h(U)$ has full Lebesgue measure, we see that $\hat{U}'$ is an open set satisfying all four items (1)-(4).

It remains to find $\hat{U}$ as described, \ie satisfying (1)-(3) and such that $\pi(\hat{U})$ is dense. We construct it recursively. First fix a countable dense set $\{x_n\}_{n\in \N}$ in $\T^2$, with $x_i\neq x_j$ for $i\neq j$, and let $\{v_n\}_{n\in \N}$ be an enumeration of $\Z^2$. Let $\hat{x}_n$ be the unique element of $\pi^{-1}(x_n)\cap [0,1)^2$.
We will find $\hat{U}$ such that $\pi(\hat{U})$ is connected, simply connected and contains $\{x_n\}_{n\in \N}$ (so it is dense), and moreover, such that for each $k$ there is $n_k$ such that $\hat x_{n_k}+v_k\in \hat{U}$. This will imply that $\hat{U}$ intersects $[0,1)^2 + v_k$, guaranteeing the required properties.

\emph{Step 0.} Let $n_0=1$. By Lemma $\ref{lem:ex-aux}$, we may choose a compact arc $\gamma_0$ joining $\hat{x}_0$ to $\hat{x}_{n_0}+v_0$ such that $\pi\circ \gamma_0$ is injective (\ie it is a simple arc). This implies that we can find a neighborhood $\hat{U}_0$ of $\gamma_0$ such that $\pi(\cl{\hat{U}}_0)$ is homeomorphic to a closed disk $\ol{\D}$. 

\emph{Step 1.} Since $\cl{\hat{U}}_0$ is disjoint from its integer translates, one can find a simply connected neighborhood $V$ of $\cl{\hat{U}}_0$ that is disjoint from $V+v$ for any $v\in \Z^2_*$. Let $n_1$ be the smallest integer such that $x_{n_1}\notin \pi(\cl{\hat{U}_0})$. Lemma \ref{lem:ex-aux} implies that there is a compact arc $\gamma_1$ joining $\hat{x}_0$ to $\hat x_{n_1}+v_1$ such that $\pi\circ \gamma_1$ is simple.
If $\gamma_1$ intersects $\cl\hat{U}_0+v$ for some $v\in \Z^2_*$, we can modify $\gamma_1$ by removing the subarc between the first intersection $z$ of $\gamma_1$ with $\ol{V}+v$ and the last intersection $z'$ of $\gamma_1$ with $\ol{V}+v$, and replacing it by a simple arc joining $z$ to $z'$ and contained (except for its endpoints) in the annulus $(V+v)\sm (\cl\hat{U}_0+v)$. The new arc still projects to a simple arc in $\T^2$, and it does not intersect $\hat{U}_0+v$. Since there are at most finitely many choices of $v\in \Z^2$ such that $\cl\hat{U}_0+v$ intersects $\gamma_1$, by repeating this process we may assume that $\gamma_1$ does not intersect $\cl\hat{U}_0+v$ for any $v\neq(0,0)$.

This implies that $K=\cl{\hat{U}_0}\cup \gamma_1$ is a compact set and $K+v$ is disjoint from $K$ for all $v\in \Z^2_*$. If $W$ is a sufficiently small connected neighborhood of $K$ then $\ol{W}$ is disjoint for $\ol{W}+v$ for any  $v\in \Z^2_*$ as well. Define $\hat{U}_1$ as the union of $W$ with all bounded components of $\R^2\sm W$. Then $\hat{U}_1$ is open, simply connected and it still holds that $\cl\hat{U}_1$ is disjoint from $\cl\hat{U}_1+v$ for any $v\in \Z^2_*$.

\emph{Step k+1.} Suppose we have defined a sequence $\hat{U}_0\subset \hat{U}_1\subset \cdots \subset \hat{U}_k$ of open simply connected sets and a sequence $n_0<n_1<\cdots<n_k$ such that, for $i\leq k$, 
\begin{itemize}
\item $\hat{U}_i$ is open, connected, simply connected;
\item $\cl\hat{U}_i$ is disjoint from $\cl\hat{U}_i+v$ for all $v\in \Z^2_*$;
\item $x_j\in \cl\pi(\hat{U}_i)$ for $0\leq j<n_i$;
\item $\hat{x}_{n_i}+v_i\in \hat{U}_i$.
\end{itemize}

To obtain $\hat{U}_{k+1}$ we repeat what was done in Step 1: fix a larger simply connected neighborhood $V$ of $\hat{U}_k$ such that $V\cap (V+v)=\emptyset$ for $v\in \Z^2$, define $n_{k+1}$ as the smallest integer such that $x_{n_{k+1}}\notin \pi(\cl{\hat{U}_k})$, and choose an arc $\gamma_{k+1}$ joining $\hat{x}_0$ to $\hat{x}_{n_{k+1}}$. We may assume that $\gamma_{k+1}$ is disjoint from $\cl{\hat{U}_k}+v$ by modifying $\gamma_{k+1}$ as before: for each $v\in \Z^2_*$ such that $\gamma_{k+1}$ intersects $\cl{\hat{U}_k}+v$ (note that there are finitely many such values of $v$), replace the part of $\gamma_{k+1}$ between the first and last intersection of $\gamma_{k+1}$ with $\ol{V}
+v$ by a simple arc in $(V+v)\sm (\cl\hat{U}_k+v)$ joining the same two points.

Finally, we let $K = \cl{\hat{U}_k}\cup \gamma_{k+1}$, we choose a small connected neighborhood $W$ of $K$ such that $\ol{W}\cap (\ol{W}+v)=\emptyset$ for $v\in \Z^2_*$, and we define $\hat{U}_{k+1}$ as the union of $W$ with all the bounded components of $\hat{U}_{k+1}$. One readily verifies that $\hat{U}_{k+1}$ has the required properties to continue 
with the recursion.

Defining $\hat{U}= \bigcup_{k\in \N} \hat{U}_k$, it is clear from the construction that it has the required properties, completing the proof.
\end{proof}

\section{Construction of the examples: Theorems \ref{th:example} and \ref{th:bounded}}

Let us first introduce some definitions from \cite{katok}. We say that a sequence $\rho = (\rho_n)_{n\in \N}$ of real valued continuous functions of the closed unit disc $\ol{\D}$ is admissible if each $\rho_n$ is positive in the open disk $\D$, and we define $C^\infty_\rho(\ol{\D})$ as the set of all $C^\infty$ functions $h\colon \ol{\D}\to \R$ such that for all $n\geq 0$ there is $\epsilon_n>0$ such that at points within a distance $\epsilon_n$ of $\bd \D$, the partial derivatives of $h$ of orders at most $n$ are bounded by $\rho_n$; \ie, whenever $\norm{(x_1,x_2)}\geq 1-\epsilon_n$ and $i+j\leq n$, $i,j\geq 0$, we have 
$$\abs{\partial_{x_1}^i\partial_{x_2}^j h(x_1,x_2)}\leq \rho_n(x_1,x_2).$$ 

We may now define $\diff^\infty_\rho(\ol{\D})$ as the set of all $C^\infty$ diffeomorphisms $f\colon \ol{\D}\to \ol{\D}$ such that the two coordinate functions of $(f-\id)$ lie in $C^\infty_\rho(\ol{\D})$.

\begin{theorem}[\cite{katok}]\label{th:katok} For every admissible sequence of functions $\rho$ on $\ol{\D}$, there exists a diffeomorphism $g\in \diff^\infty_\rho(\ol{\D})$ which preserves the normalized Lebesgue measure such that $(f,\frac{1}{\pi}\Leb|_{\ol{\D}})$ is metrically isomorphic to a Bernoulli shift.
\end{theorem}

We will use the following fact, which is essentially Proposition 1.1 of \cite{katok}, with the slight difference that for us $\D$ denotes the open disk instead of the closed disk (but it is proved exactly in the same way).

\begin{proposition}\label{pro:ex-adm} Let $S$ be a surface and $H\colon \D\to U\subset S$ be a $C^\infty$ diffeomorphism. Then there exists an admissible sequence of functions $\rho$ such that for every $\til{f}\in \diff^\infty_\rho(\ol{\D})$, the map $f\colon S\to S$ defined by
$$f(z)=\begin{cases} H\til{f}H^{-1}(z) &\text{ if }z\in U, \\ z& \text{ otherwise} \end{cases}$$
is a $C^\infty$ diffeomorphism.
\end{proposition}

\subsection{Proof of Theorem \ref{th:example}} \label{sec:ex-proof}

Consider the set $\hat{U}$ from Theorem \ref{th:weirdopen}. There exists a $C^\infty$ (orientation preserving) diffeomorphism $h\colon \D\to \hat{U}$ such that $h^*(\omega|_{\hat{U}}) = \frac{1}{\pi}\omega|_\D$, where $\omega = dx\wedge dy$ is the Lebesgue volume form. This can be seen as follows: from the Riemann mapping theorem we know there is a $C^\infty$ diffeomorphism $\til{h}\colon \D \to \hat{U}$ preserving orientation.  Note that $\int_{\hat{U}} \omega =1$, so the $C^\infty$ volume form $\alpha = \til{h}^*(\omega|_{\hat{U}})$ on $\D$ satisfies $\int_\D \alpha =1 = \int_\D \frac{1}{\pi}\omega|_\D$. Thus, the main theorem of \cite{greene-shiohama} guarantees that there is a $C^\infty$ diffeomorphism $\phi \colon \D \to \D$ such that $\phi^*(\alpha) = \frac{1}{\pi}\omega|_\D$.
Letting $h = \til{h}\phi$, it follows that $h^*(\omega|_{\hat{U}}) = \frac{1}{\pi}\omega|_{\D}$, as we wanted.

Let $U=\pi(\hat{U})$, and let $\rho$ be the admissible sequence obtained from Proposition \ref{pro:ex-adm} using $H = \pi|_{\hat{U}}h$. Note that $\phi|_{\hat{U}}$ is injective, and if $\omega_{\T^2}$ denotes the lebesgue volume form on $\T^2$ (defined by $\omega'(z) = \omega(\hat{z})$ where $\hat{z}\in \pi^{-1}(z)$), then $H^*(\omega'|_U) = \frac{1}{\pi}\omega|_{\D}$. This in particular implies that if $\til{f}\colon \D\to \D$ is an area-preserving diffeomorphism, then $H\til{f}H^{-1}\colon U\to U$ is also area-preserving.

Theorem \ref{th:katok} guarantees that there is a $C^\infty$ diffeomorphism $\til{f}\colon \ol{\D}\to \ol{\D}$ preserving the normalized Lebesgue measure $\frac{1}{\pi}\Leb|_{\ol{\D}}$, and such that $(\til{f},\frac{1}{\pi}\Leb_{\ol{\D}})$ is Bernoulli. The map $f$ defined as in Proposition \ref{pro:ex-adm} is a $C^\infty$ diffeomorphism such that $f|_U=H\til{f}H^{-1}$ (and $f$ fixes points of the complement of $U$), and from the fact that $U$ has full Lebesgue measure and our previous observations it follows that $f$ is area-preserving. Moreover, $f$ is metrically isomorphic to $\til{f}$ and therefore to a Bernoulli shift. In particular, $f$ is ergodic, and so Lebesgue almost every point in $\T^2$ has a dense positive $f$-orbit. If $\hat{f}\colon \R^2\to \R^2$ is the lift of $f$ that that leaves $\hat{U}$ invariant, then $\pi|_{\hat{U}}$ is a homeomorphism onto $U$ which preserves Lebesgue measure such that $\pi|_{\hat{U}}\hat{f} = f\pi|_{\hat{U}}$, so that the positive $\hat{f}$-orbit of Lebesgue almost every point in $\hat{U}$ is dense in $\hat{U}$. In particular, since $\hat{U}$ accumulates in all directions at infinity, this implies that the orbit of Lebesgue almost every point in $\hat{U}$ accumulates in all directions at infinity. Since this also holds for any integer translation of $\hat{U}$, and the complement of $\bigcup_{v\in \Z^2} \hat{U}+v$ has null Lebesgue measure, we conclude the last claim of the theorem.

It remains to show that $\rho(\hat{f}) = \{(0,0)\}$.  Let $w\in \R^2$ be an extremal point of $\rho(\hat{f})$. From \cite{m-z} it is known that for any such extremal point one can find an $f$-recurrent point $z\in \T^2$ such that if $\hat{z}\in \pi^{-1}(z)$ then $(\hat{f}^n(\hat{z})-\hat{z})/n\to w$ as $n\to \infty$. Suppose $\hat{z}$ is not fixed, so that $\hat{z}\in \hat{U}$, and fix $\epsilon>0$ such that $B_\epsilon(\hat{z})\subset \hat{U}$. Since $z$ is recurrent, there is a sequence $n_k\to \infty$ such that $f^{n_k}(z)\in B_\epsilon(z)$, so that there are $v_k\in \Z^2$ such that $\hat{f}^{n_k}(\hat{z})\in B_\epsilon(\hat{z})+v_k$. But since $\hat{z}\in \hat{U}$ it follows that $f^n(\hat{z})\in \hat{U}$ for any $n\in \Z$, and since $B_\epsilon(\hat{z})\subset \hat{U}$ and $\hat{U}$ is disjoint from its integer translations, one concludes that $v_k=(0,0)$. Thus $\norm{\smash{(\hat{f}^{n_k}(\hat{z})-\hat{z})/n_k}}\leq \epsilon/n_k \to 0$ as $k\to \infty$, which implies that $w=(0,0)$. We conclude that the only extremal point of $\rho(\hat{f})$ is $(0,0)$, and since the rotation set is convex it follows that $\rho(\hat{f})=\{(0,0)\}$, as we wanted to show.\qed

\subsection{Idea of the proof of theorem \ref{th:bounded}} The proof is straightforward using the same ideas from the previous section: we consider the open set $U$ from Theorem $\ref{th:weirdopen}$, and define $f$ such that $f(U)=U$ and $f$ fixes $\T^2\sm U$ pointwise. Instead of using the Bernoulli system from Theorem \ref{th:katok} as a model, we use an elliptic diffeomorphism: $\til{f}\colon \ol{\D}\to \ol{\D}$ defined in polar coordinates as $\til{f}(r,\theta) = (r, \theta+\phi(r))$ where $\phi\colon[0,1]\to \R$ is a $C^\infty$ increasing function such that $\phi(0)=0$, $\phi(1)=1$ and $0<\phi(\theta)<1$ when $0<\theta<1$. If $\phi$ is chosen appropriately, we may guarantee that $\til{f}\in \diff^\infty_\rho(\ol{\D})$ for the admissible sequence $\rho$ given by Proposition \ref{pro:ex-adm}, and so we are able to finish the construction as in the previous section. The resulting map $f$ on $\T^2$ has a sequence of invariant topological circles accumulating on the boundary of $U$ and such that the restriction of $f$ to each circle is an irrational rotation. Since $\hat{U}$ (and its boundary) intersects every set of the form $[0,1]^2+v$ with $v\in \Z^2$, the claims from the theorem follow easily.\qed

\bibliographystyle{koro} 
\bibliography{irrotational}

\end{document}